\documentclass[11pt]{amsart}

\usepackage{amsgen,amsthm,amstext,amsbsy,amsopn,amsfonts,amssymb, enumerate}

\usepackage{hyperref}
\usepackage{amsmath}
\usepackage[usenames]{color}
\usepackage{xcolor}
\usepackage[normalem]{ulem}

\newcommand\ggo{{\mathfrak g}}

\newcommand\kk{{\mathfrak k}}

\newcommand\RR{\mathbb R}

\newcommand\ad{\operatorname{ad}}
\newcommand\Ad{\operatorname{Ad}}

\newcommand\Iso{\operatorname{Iso}}

\theoremstyle{plain}

\newtheorem{thm}{Theorem}[section]
\newtheorem{lem}[thm]{Lemma}
\newtheorem{prop}[thm]{Proposition}

\theoremstyle{definition}

\newtheorem{rem}[thm]{Remark}

\begin{document}

\title[The mean curvature flow of subgroups on  Lie groups]{The mean curvature flow of subgroups on Lie groups of dimension three}

\author[Arroyo]{Romina M. Arroyo}
\address{FAMAF \& CIEM - Universidad Nacional de C\'ordoba. Av  Medina Allende s/n - Ciudad Universitaria.
ZIP:X5000HUA - Cordoba, Argentina}
\email{arroyo@famaf.unc.edu.ar}
\thanks{RA has been partially supported by CONICET (PIP 11220210100243CO), ANPCyT (PICT-2020-SerieA-02432) and SeCyT-UNC (33620230100675CB)}

\author[Ovando]{Gabriela P. Ovando}
\address{Departamento de Matem\'atica, ECEN - FCEIA, Universidad Nacional de Rosario. Pellegrini 250, 2000 Rosario, Santa Fe, Argentina}
\email{gabriela@fceia.unr.edu.ar}
\thanks{GO has been partially supported by CONICE (PIP 11220220100615CO), ANPCyT (PICT-2021-GRF-TII-0037) and SeCyT-UNR (80020210200079)}

\author[S\'aez]{Mariel S\'aez}
\address{Facultad de Matem\'atica, P. Universidad Cat\'olica de Chile, Av. Vicuna Mackenna 4860, 690444 Santiago, Chile.}
\email{mariel@uc.cl}

\thanks{{\it (2000) Mathematics Subject Classification}:  53A10, 53C42, 53C44, 22E25, 22F30}

\thanks{{\it Key words and phrases}:   solitons, mean curvature flow, solvmanifolds}



\begin{abstract}  
In this work we study the existence of solutions to  the Mean Curvature Flow for which the initial condition has the structure of a two-dimensional Lie subgroup within a  Lie group of dimension three. 
We consider Lie groups with a suitable left-invariant metric and 
first observe that if the Lie group is unimodular, then every  Lie subgroup is a minimal surface (hence a trivial solution). For this reason we focus on non-unimodular Lie groups, finding the evolution of every Lie subgroup of dimension 2 (within a 3 dimensional Lie group). These evolutions are self-similar for abelian subgroups (i.e. evolve by isometries), but not self-similar in the other cases. We conclude by showing that normal Lie subgroups of codimension one in ambient Lie groups of any dimension evolve self-similarly.

\end{abstract}

\maketitle

\setcounter{tocdepth}{1}
\tableofcontents

\section{Introduction}\label{sec:intro}

\medskip

The Mean Curvature Flow is a geometric equation that arises as the gradient flow of the area functional. More precisely, given an $(n+1)$-dimensional ambient Riemannian manifold  $(G,g)$, a $k$-submanifold  $K$ (of dimension $k$) and an embedding
$\varphi_0: K\to G$ ,  we are interested in finding a family  of embeddings $\varphi(\cdot, t): K\to G$ that is continuous in $t$,  $\varphi(\cdot, 0)=\varphi_0$  and that for each time $t>0$  satisfies 
$$\frac{\partial \varphi}{\partial t}=\vec{H},$$
where $\vec{H}$ is the mean curvature vector associated to the embedding $\varphi(\cdot, t)$. When the co-dimension is 1 (i.e. $k=n$) this can be equivalently written as 
\begin{equation} 
 g\left(\frac{\partial \varphi}{\partial t}, \nu\right)=-H, \label{eq:MCF}
  \end{equation}
where $H$ and $\nu$  are the (scalar) mean curvature and a  unit normal  of $K$, respectively.  \\

This equation has been extensively studied when the ambient manifold is $\mathbb{R}^{n+1}$ with the usual metric (see for instance \cite{ABGL} and the numerous references therein), showing that singular behavior may arise during the evolution. In many situations, these singularities are modeled by self-similar solutions, that in Euclidean space correspond to evolutions by translation and by homothety; many
 works are aimed at understanding these special solutions (e.g. \cite{AL}, \cite{BLT}, \cite{HIMW}, \cite{HMW}, \cite{CHH}, among many others).

 In the context of this paper, a relevant example is the \emph{Grim Reaper} solution. 
This is an explicit curve that solves  \eqref{eq:MCF} when $n = 1$ and can be readily 
derived by assuming a translating graphical solution of the form $(x, u(x) + t)$. 
Moreover, this solution extends naturally to higher dimensions by keeping the 
remaining directions static, for example $(x, y, u(x) + t)$. Another tractable 
class of solutions that has been studied consists of rotationally symmetric 
self-similar solutions, for which  \eqref{eq:MCF} reduces to an ordinary differential 
equation.\\

For non-Euclidean ambient spaces, there are fewer works and concrete examples (some of them are \cite{Huisken1986ContractingCH} and \cite{Andrews}); however 
 it is expected to observe different phenomena than in the Euclidean case.
 In order to achieve a better understanding of these differences, it is relevant to construct new examples of evolutions and to this end many authors have studied self-similar solutions in this context, since they are easier to understand than general evolutions,  due to their additional structure. In this setting
a self-similar solution is an evolution that is determined by the flow of a conformal Killing vector field $V$, i.e. a vector field that  satisfies
$\mathcal{L}_V g= C g,$ for some function $C$ (see for instance \cite{CoMaRi} for a more detailed description).
 Note that in  the particular case that $C=0$  the vector field $V$ is Killing, that in $\mathbb{R}^{n+1}$  corresponds to fields generated by rigid motions, in particular including solutions that evolve by translations; for this reason we will refer to them  as translators even in more general settings. It is also worth observing that  the Lichnerowicz conjecture states  that a Riemannian manifold that admits 
 conformal transformations with  $C\ne 0$ (known as essential) is conformally equivalent to the standard sphere or the Euclidean space.   This conjecture was proved by works of  M. Obata \cite{Obata} and J. Ferrand \cite{Fer} in the compact case and in  general in \cite{A,A2,Fer,F}, hence in the cases described by the present work
  we will be able to reduce our study to Killing vector fields (since our spaces are not conformal to the sphere nor Euclidean space).\\
 
Among the most relevant constructions of self-similar solutions in non-Euclidean settings we would like to mention constructions in Hyperbolic Space $\mathbb{H}^n$
(see for instance \cite{LRS},  \cite {BO2}, \cite{BuenoLopez1},  and references therein), $\mathbb H^m \times  \mathbb R$ (e.g. \cite{Bueno},  \cite{BO}, \cite{LP} and \cite{BuenoLopez2} ),
$Sol_3$ (see \cite{Pi1}) and the
Heisenberg group endowed with a left invariant metric  (see \cite{Pi2}). A common feature of these constructions is that in order to have a simpler analytical structure
further symmetries are considered-- either rotationally invariant  solutions (e.g. \cite{LRS, Bueno, BO}) or that they are solutions generated by a curve (we refer to them as {\it Grim-Reaper solutions}, as in \cite{BuenoLopez2, Pi1, Pi2}). 

  The aforementioned examples also share the fact that the ambient manifolds can be presented as homogeneous Riemannian spaces. This class of manifolds possesses a rich underlying algebraic structure and the main goal of our work is to explore constructions that exploit that fact. In contrast with previous examples, we aim at constructions 
from a purely  algebraic perspective, an
 idea that has been successfully used for other geometric flows (for instance for the Ricci flow in \cite{BL1,BL2} and in references therein). 
More precisely, we assume that the ambient manifold $G$ is determined by a Lie group of dimension 3 endowed with a left invariant metric and we look for solutions that emanate from 
 subgroups of $G$ of dimension 2. In the Euclidean setting these examples are all trivial, since they are planes that are minimal surfaces and, hence, static solutions to the flow. In contrast, in this work we compute the evolution of  several non-trivial translators that have a subgroup as initial condition.
 Moreover, these translating solutions allow us to conjecture and compute  the evolutions by Mean Curvature Flow for every  initial condition that is given by a  subgroup of dimension 2 (within a 3 dimensional Lie group), even in the case  they are not translators.\\

Before stating  our  results we recall that $3$-dimensional real Lie algebras, their associated Lie groups and subgroups are already classified in the literature (see for instance \cite{MP}). However, this classification is  often written up to certain normalization;
for instance,
the classification of Lie algebras obtained by Milnor \cite{Mi} as a semidirect product by $\mathbb R$ of $\mathbb R^2$ is described by fixing the trace of the operator acting on $\mathbb R^2$. In those works that normalization is sufficient, since the goal is to describe these objects up to group isomorphism. Instead, we make use of the Jordan decomposition to distinguish them. That choice is justified by our interest in  exploring  in future work the moduli space of solutions to Mean Curvature Flow as
  we  vary the parameters that define the ambient space and in consequence we explicitly compute  Lie groups and subgroups with  all the parameters. \\

A more precise description of the classification of Lie Groups and their subgroups will be discussed in Section \ref{sec:preliminaries}, but it is relevant to observe that Lie algebras can be divided into two classes: the ones associated to a unimodular Lie group and those that are non-unimodular. In the former case, when the group is endowed with a left-invariant metric, all subgroups define  minimal surfaces (as in Euclidean space), while for non-unimodular  groups this is not the case and we have three families that can be studied (they will be described precisely in Section \ref{sec:preliminaries}). In particular,  $Sol_3$  and the 
Heisenberg group are unimodular, hence all our solutions are trivial (for  $Sol_3$ our examples should be compared to  \cite{LoMo} and for the Heisenberg group with \cite{FMP}), while $\mathbb{H}^3$ and $\mathbb{H}^2\times \mathbb{R}$ are non-unimodular; we show that the latter contain subgroups that produce non-trivial solutions (for further comparisons see Remark \ref{comparisonwithknown}).\\

Our first theorem can be then summarized as follows 

\begin{thm} \label{thm:mainintro}
	
 	Let $G$ be a non-unimodular  Lie group of dimension three equipped with a suitable left-invariant metric. Let $K$ denote an abelian Lie subgroup of $G$ with  non-vanishing mean curvature, then its mean curvature is constant and there is a family of translations $\varphi(\cdot, t)$ such that $\varphi(K, t)$ is a solution to \eqref{eq:MCF}.
\end{thm}

 In fact,  Theorem \ref{thm:mainintro}  is a consequence of Theorem \ref{main theorem}, that contains a complete list of possible translators. However, in order to properly state that result several  definitions  of Section \ref{sec:preliminaries}  are necessary. We briefly comment on the hypotheses: we first fix a left-invariant metric on $G$  for which the basis defined  in Section \ref{sec:preliminaries}  is orthonormal (see \eqref{metrics} for the explicit metrics). We then observe that
on unimodular groups $G$ all subgroups have vanishing mean curvature, that  is they are minimal surfaces, hence static solutions to the Mean Curvature Flow. In contrast, for non-unimodular groups, their subgroups have constant mean curvature that is not necessarily 0.  Among these subgroups, the evolution of the
 abelian ones is given by translations. For the non-abelian case we have the following result.

\begin{thm} \label{thm:mainintro2}
	
 	Let $G$ be a non-unimodular  Lie group of dimension three equipped with a suitable  left-invariant metric. Let $K$ denote a non-abelian subgroup of $G$ with  non-vanishing mean curvature, then its mean curvature is constant and there is an explicit solution to \eqref{eq:MCF} that exists for all time, for each time $t$ has constant mean curvature and it is not a self-similar solution.
	
\end{thm}	
 
  As in the previous case, this result is a consequence of Theorem \ref{thm:main2}  that contains a precise description of the solutions to \eqref{eq:MCF} with initial condition given by a non-abelian subgroup.\\

A natural question that arises is whether in higher dimensions subgroups that evolve by translation also agree with the abelian subgroups.
In the last section of this work we show that this is not the case, since we can exhibit an example of a four dimensional Lie group with a subgroup of dimension three that is not abelian and evolves by translation.  We can frame this example in a more general class that is given by the following result.

\begin{thm} \label{thm:mainintro3} 
Let $G$ be a non-unimodular  Lie group of dimension $n+1$ equipped with a left-invariant metric. Let $K$ denote a normal Lie subgroup of $G$ of co-dimension 1 with unit normal vector given by $\nu$. Then its mean curvature is a constant $H_0$ and the one parameter family
 $$\varphi(p,t)=p \exp (t H_0 \nu) \quad p\in K$$
 gives a solution to \eqref{eq-EvolutionKilling1}, inducing a self-similar solution to  \eqref{eq:MCF} 

\end{thm}

It is relevant to remark that this result generalizes Theorem \ref{thm:mainintro}, since the two dimensional normal groups are the abelian ones. We still keep two separate proofs (although it is not strictly necessary), since the computational nature of the proof  of Theorem  \ref{main theorem} allow us to have concrete examples, that provide us with the ansatz we use to show
Theorem \ref{thm:main2}. In addition, it is relevant to observe that Theorems \ref{thm:mainintro} and  \ref{thm:main2} give a full characterization of the evolutions of co-dimension 1 Lie subgroups within  3 dimensional  non-unimodular  Lie groups, while Theorem \ref{thm:mainintro3}  characterizes only a subfamily of possible subgroups.\\

The paper is organized as follows: Section \ref{sec:preliminaries} includes all the background on 3 dimensional Lie groups, their subgroups and  computations of the associated geometric quantities  in this context. In Section \ref{sec:solitons} we impose an evolution determined by a Killing vector field and find solutions to \eqref{eq:MCF}. Section \ref{sec:solutions} includes the explicit solutions and the
proofs of Theorems \ref{thm:mainintro} (resp. Theorem  \ref{main theorem}) and \ref{thm:mainintro2} (resp. Theorem \ref{thm:main2}).  We finish in Section \ref{counterex} by showing the evolution by translation of non-abelian subgroup of dimension 3 within a four dimensional ambient Lie group and proving Theorem \ref{thm:mainintro3}.


\section{Preliminaries on Lie groups of  dimension three} \label{sec:preliminaries}

The goal of this section is to provide  preliminaries on  Lie groups equipped with a left-invariant metric that will be necessary to study our solutions to Mean Curvature Flow. We also discuss some existing classification results in dimension 3 and compute all the families of relevant subgroups in our context.


Let $(G,g)$ be a Lie group equipped with a left-invariant metric. We denote by $\mathfrak{X}(G)$ the set of differentiable vector fields on $G$  and let $\nabla$ be the Levi-Civita connection   associated to $g$. The Koszul formula for left-invariant vector fields $U,V\in \ggo$ establishes
$$2 \nabla _U V = [U,V]- \ad_U^t(V) - \ad_V^t(U),  
$$
where  $\ggo$ is the Lie algebra associated to $G$,  $\ad_U:\ggo\to \ggo$  is the adjoint endomorphism of  $\ggo$ (given by $\ad_U(V)=[U,V]$)
and  $\ad_U^t$ denotes its transpose  with respect to the metric $g$.\\

Let $K$ denote an $n$-dimensional submanifold in $(G,g)$ of codimension one. Let $TK$ denote its tangent space equipped with the induced metric  
$\bar{g}$ defined by
$$\overline{g}(X,Y)=g(X,Y) \hbox{ for } X,Y \in TK.$$

Around any point $p$ in $K$ one has 
$$T_pG= T_pK \oplus \RR \nu,$$
where $\nu$ denotes a unit normal vector field to $T_pK$. \\

Assume $\nu \in \mathfrak{X}(G)$. The second fundamental form is a  bilinear extrinsic geometric quantity 
 $\alpha: \mathfrak{X}(K)\times \mathfrak{X}(K)\to\mathbb{R}$  defined by
$$\alpha(X,Y)=g(\nabla_X \nu,Y)=-g(\nu,\nabla_XY).$$

Note that since $g(\nu, \nu)=1$ we have that for every $X\in \mathfrak{X}(K)$,
$$ g(\nabla_X\nu, \nu)=0.$$
The Weingarten map (also called shape operator) $A_\nu: \mathfrak{X}(K)\to \mathfrak{X}(K) $ is defined by
$$A_\nu (X)=\nabla_X\nu.$$
With these definitions we have  $$\alpha(X,Y)=g(A_\nu X,Y)=- g(\nu,\nabla_XY).$$
In addition, for $X,Y\in \mathfrak{X}(K)$ one has
$$\nabla_X Y = \overline{\nabla}_XY-\alpha(X,Y) \nu, $$
where  $\overline{\nabla}$ the Levi-Civita connection on $K$.

Since the Levi-Civita connection is torsion free we have
$\nabla_XY-\nabla_YX=[X,Y]$, for all $X,\,Y\in \mathfrak{X}(G),$ 
 and it is well-known that this implies that $A_\nu$ is self-adjoint:
$$g(A_{\nu}X,Y)=g(X, A_{\nu}Y).$$

The trace of $A_\nu$, denoted by $H$, is known as the {\em scalar mean curvature}, or simply {\em mean curvature}, and it is independent of parametrization. 
Submanifolds that satisfy $H\equiv 0$ are known as {\em minimal.}

Note that for an orientable manifold $\nu$ can be chosen up to a sign, but $\vec{H}=-H \nu$ is independent of this choice. Nonetheless, there are different sign conventions for the definition of $H$ and 
 the choices in Equation \eqref{eq:MCF} are such that the solution locally decreases area, which, from the PDE point of view implies  \eqref{eq:MCF} is a parabolic partial differential equation, or equivalently $\vec{H}=\Delta_{\bar{g}} \varphi(\cdot, t)$, where $\Delta_{\bar{g}} $ is the Laplace-Beltrami operator with respect to the induced metric $\bar{g}$ on $K$.
\\




A Lie group $G$ is said to be {\em unimodular} if and only if
$$|det \,\Ad_g|=1, \qquad  \forall g\in G$$
where $\Ad$ denotes the adjoint representation of $G$. This implies at the Lie algebra level
 $$tr\,\ad_X=0$$ for any $X\in \mathfrak g$, with equivalence whenever the  Lie group $G$ is connected. Note that here $\ad$ denotes the adjoint representation at the algebra level, while $\Ad$ is defined on $G$.  \\

In this work we are going to study submanifolds  $K$  that are also subgroups of $G$. Recall that we are assuming that $G$ is equipped with a left-invariant metric, which induces  a (left-invariant) metric on $K$. The tangent spaces $T_p G$ and $T_pK$ at a point $p\in K$ are naturally identified with  corresponding Lie algebras, that we denote  $\mathfrak g$
and  $\mathfrak k$, respectively.\\


Due to the algebraic structure, the following lemma allow us to compute the mean curvature of $K,$ showing that it is constant. 

\begin{lem}\label{lema1} Let $K$ be a subgroup of a Lie group $G$ and let $g$ be a left-invariant metric on $G$. Consider on $K$ the induced metric. Then the mean curvature of $K$ is  a constant given by
	$$H = -tr\ad_{\nu}, \quad\mbox{ for $\nu$ unitary normal vector field}.$$ 
	In particular, whenever $G$ is unimodular, the Lie subgroup $K$ is minimal. 
\end{lem} 
\begin{proof} Let $\{X_i\}$ be a orthonormal basis of  $\mathfrak k$ (the Lie algebra of $K$).  From the Koszul formula one gets that
	$$g(\nabla_{X_i} X_i, \nu)=g([\nu, X_i], X_i).$$
	Thus, the mean curvature can be computed in terms of the basis $X_1, X_2, \hdots X_n, \nu$ as follows
	$$tr A_{\nu}=\sum_i g(A_{\nu}X_i, X_i)=-\sum_i g(\nu, \nabla_{X_i}X_i) = -\sum_i g([\nu, X_i], X_i)=- tr\ad_\nu.$$
	
	The second statement is a direct consequence from this formula. 	
\end{proof} 

The last concept that we will need in our study is the one of solvable Lie algebra. To that end, given a Lie algebra $\mathfrak{g}$, we first define its {\em derived series} as follows
$$D^0(\mathfrak g)=\mathfrak g, \qquad \quad D^j(\mathfrak g)=[D^{j-1}(\mathfrak g), D^{j-1}(\mathfrak g)] \quad\mbox{ for }j\geq 1.$$
If there is $J>0$ such that $D^J(\mathfrak g)=0$, the Lie algebra is called {\em solvable} and the corresponding Lie group is called a solvable Lie group. In the case of dimension 3, all non-solvable Lie algebras are unimodular and, in particular, their subgroups have zero mean curvature. \\

\subsection{Three dimensional Lie groups with codimension one subgroups}\label{sec:subgroups}

A Lie algebra over $\mathbb R$ is a real vector space  together with a Lie bracket. The classification of $3$-dimensional real Lie algebras up to isomorphims is well-known. A list of non-abelian Lie algebras spanned by vectors $E_1, E_2, E_3$ is the following (see for instance \cite{Mi,GOV}). \\

\begin{enumerate}[(i)]
	\item \label{solv1}$\mathfrak r_{3,\lambda}$: $[E_3, E_1]=\lambda E_1$, $[E_3, E_2]=E_1+\lambda E_2$, with $\lambda\in \RR$. For $\lambda=0$ one has the Heisenberg Lie algebra,  usually denoted by $\mathfrak{h}_3$; 
	\item  \label{solv2} $\mathfrak r_{3,\lambda_1, \lambda_2}$: $[E_3, E_1]=\lambda_1 E_1$, $[E_3, E_2]=\lambda_2 E_2$, for $\lambda_1\in \mathbb R-\{0\}, \lambda_2\in \mathbb R$;
	\item \label{solv3} $\mathfrak r_{3,\lambda}'$: $[E_3, E_1]=\lambda E_1-E_2$, $[E_3, E_2]=E_1+\lambda E_2$, $\lambda\geq 0$;
	\item \label{simple1}  $\mathfrak{sl}_2$: $[E_3, E_2]=E_1, [E_1, E_3]=2E_3, [E_1, E_2]=-2E_2$,
	\item \label{simple2} $\mathfrak{su}_{2}$: $[E_3, E_2]= - E_1$, $[E_1, E_2]=E_3, [E_1, E_3]=- E_2$. 
\end{enumerate}

\medskip

\begin{rem}  The Riemannian manifolds $\mathbb H^2 \times \mathbb R$, $\mathbb H^3$, $Sol$ and  $Nil$ 
admit a simply transitive 3-dimensional subgroup of isometries. Consequently, they 
are homogeneous Riemannian spaces and each may be realized as a Lie group endowed with a suitable left-invariant 
Riemannian metric. The Lie algebras associated with the underlying Lie groups of these geometries are 
$\mathfrak{r}_{3,1,0}$, $\mathfrak{r}_{3,1,1}$, $\mathfrak{r}_{3,1,-1}$ and the Heisenberg Lie algebra 
$\mathfrak{h}_{3}$, respectively. The unique non-commutative structure for Euclidean geometry, on the other hand, arises from the Lie algebra $\mathfrak{r}'_{3,0}$, which corresponds to the Lie algebra of the group of rigid motions of $\mathbb R^2$.\\

\end{rem}

We first discuss the cases  \eqref{solv1}, \eqref{solv2} and \eqref{solv3}, that
  correspond to solvable Lie algebras. Note that these algebras
 can be decomposed into the direct sum of vector subspaces $\mathbb R \oplus \mathbb R^2$, where $\mathbb R^2$ is the abelian ideal spanned by $E_1,E_2$. Their corresponding Lie groups can be represented, respectively, as the following groups of matrices.
 
$$
S_{3,\lambda}: \left(\begin{matrix}
e^{\lambda z} & ze^{\lambda z} & x\\
0 & e^{\lambda z} & y\\
0 & 0 & 1 \end{matrix}
\right), \quad  S_{3,\lambda_1, \lambda_2}: \left(\begin{matrix}
e^{\lambda_1 z} & 0 & x\\
0 & e^{\lambda_2 z} & y\\
0 & 0 & 1 \end{matrix}
\right),$$
$$ \quad S_{3,\lambda}': \left(\begin{matrix}
e^{\lambda z} \cos(z) & e^{\lambda z} \sin(z) & x\\
-e^{\lambda z} \sin(z) & e^{\lambda z} \cos(z) & y\\
0 & 0 & 1 \end{matrix}
\right).
$$
Moreover, they are
  exponential Lie groups that  are diffeomorphic to $\mathbb R^3$. The identification is as follows
$$(x,y,z) \quad \leftrightarrow \quad A_{(x,y,z)},$$
where $A_{(x,y,z)}$ is one of the matrices above and we can identify the corresponding operation on the Lie group. In fact, consider the matrices 
$$M_{\lambda,z}: \left(\begin{matrix}
	e^{\lambda z} & ze^{\lambda z}\\
	0 & e^{\lambda z} \\
 \end{matrix}
\right), \quad  M_{\lambda_1, \lambda_2}: \left(\begin{matrix}
	e^{\lambda_1 z} & 0 \\
	0 & e^{\lambda_2 z} \\
 \end{matrix}
\right), \qquad R_{\lambda, z}= \left(\begin{matrix}
e^{\lambda z} \cos(z) & e^{\lambda z} \sin(z)\\
-e^{\lambda z} \sin(z) & e^{\lambda z} \cos(z)
\end{matrix}\right),
$$
and let $v_i=(x_i,y_i)$ for $i=1,2$, then one has that the group operation in each case is given as follows
\begin{equation}\label{product}
\begin{array}{crcl}
S_{3,\lambda}: & (v_1, z_1)(v_2,z_2) & = &(v_1+M_{\lambda,z_1}v_2, z_1+z_2),\\
S_{3,\lambda_1, \lambda_2}: & (v_1, z_1)(v_2,z_2) & = & (v_1+M_{\lambda_1,\lambda_2,z_1}v_2,  z_1+z_2),\\
S_{3,\lambda}': & (v_1, z_1)(v_2,z_2) & = &(v_1+R_{\lambda,z_1}v_2, z_1+z_2).\\
\end{array}
	\end{equation}

The last two cases, (iv) and (v), correspond to simple Lie algebras. 
The Lie algebra $\mathfrak{sl}_2$ is  the set of $2\times 2$ real matrices of zero trace: 
$$\mathfrak{sl}_2=\left\{\left(\begin{matrix}
h & e\\
f & -h\\
\end{matrix}
\right)  \right\} \quad \mbox{ with Lie group } \quad \mathrm{SL}(2)=\left\{\left(\begin{matrix}
a & b \\
c & d\\
\end{matrix}
\right) :  ad-bc=1\right\}.$$ 
While $\mathfrak{su}_{2}$ is isomorphic to the set of $3\times 3$ real skew-symmetric  matrices:
$$\mathfrak{so}(3)=\left\{\left(\begin{matrix}
0 & e & f \\
-e & 0 & g\\
-f & -g & 0
\end{matrix}
\right)  \right\} \quad \mbox{ with Lie group } \quad \mathrm{O}(3)=\mbox{the set of
	 rigid motions of } \mathbb R^3.$$ 
Note that in  those cases the Lie groups are not simply connected. The corresponding simply connected Lie groups are respectively $\widetilde{\mathrm{SL}}(2,\mathbb R)$ and $\mathrm{SU}(2)$.\\

\smallskip

We finally remark that in dimension three, the unimodular Lie algebras are $\mathfrak h_3$, $\mathfrak r_{3,\lambda_1,\lambda_2}$ with $\lambda_1+\lambda_2=0$, $\mathfrak r_{3,0}'$ and the simple Lie algebras \eqref{simple1} and \eqref{simple2}.  In particular, in all these cases  when the group is endowed with a left-invariant metric their subgroups are minimal and this includes Euclidean Space, $Sol$ and the Heisenberg group.

In the coming sections we describe the 2-dimensional subgroups in each case. We remark that a different approach of this description can be found in \cite{MP}, where the authors use Milnor's work \cite{Mi}.

\subsubsection{Subgroups of the Heisenberg Lie group $S_{3,0}$.} Recall that this group is associated to the Lie algebra given by \eqref{solv1} when $\lambda=0$. Let $\mathfrak a$ be a Lie subalgebra of dimension two of $\mathfrak h_3$ spanned by vectors $X,Y$. Note that $\mathfrak a$ must be abelian since $\mathfrak h_3$ is nilpotent and subalgebras of nilpotent Lie algebras are necessarily nilpotent; furthermore,  the only nilpotent Lie algebra of dimension two is the abelian one. \\

Assume that as above $E_1, E_2, E_3$ span $\mathfrak h_3$. In the case that 
 $X,Y$ belong to the subspace spanned by $E_1,E_2$, their Lie bracket is trivial  and
  $\mathfrak a$ is abelian. Moreover, 
we can choose  $\{E_1,E_2\}$ as a basis for $\mathfrak a$. 

Assume now that $X$  is not in the subspace spanned by $E_1$ and $E_2$ (equivalently, one may assume $Y\not \in \textrm{span }\{E_1,E_2\} $), then we can write $X=a_1E_1+a_2E_2+E_3$, $Y=b_1E_1+b_2E_2$. Thus
$$0=[X,Y]=b_2 E_1 \quad \Rightarrow \quad b_2=0,$$
which enable us to take $Y=E_1$ and assume that $a_1=0$ (up to a change of basis). Thus,  basis of a Lie subalgebra of $\mathfrak h_3$ must be  one of the following.
\begin{itemize}
	\item $\{E_1,E_2\}$, 
	\item $\{E_3+aE_2, E_1\}$ with $a\in \mathbb R$.
\end{itemize}
We conclude that the Lie subgroups of $S_{3,0}$ are 
$$\mathcal K_0=\{(x,y,0): x,y\in \mathbb R\}, \quad \mathcal K_{a}=\{(x,az,z): x,z\in \mathbb R\}.$$

\subsubsection{Subgroups of $S_{3,\lambda}$.}\label{subgroupss3lambda} Here we are considering $\mathfrak r_{3,\lambda}$ (given by \eqref{solv1}) when   $\lambda\neq 0$ (since otherwise we are in the previous case). 

Let $\mathfrak a$ denote a Lie subalgebra of $\mathfrak r_{3,\lambda}$ spanned by vectors $X$,$Y$. If $X$ and $Y$ belong to the subspace spanned by $E_1$ and $E_2$ then $[X,Y]=0$ and in this case a basis of this subspace is $\{E_1, E_2\}$. 

Without loss of generality we may assume that $X$ does not belong to this subspace, or equivalently $X=E_3+c_1E_1+c_2E_2$ and  $Y=a_1E_1+a_2E_2$. We obtain
\begin{equation}[X,Y]=\lambda a_1 E_1 + a_2(E_1+\lambda E_2)=(\lambda a_1+ a_2) E_1+ \lambda a_2 E_2.\label{primerpaso}\end{equation}
On the other hand, if $[X,Y]=\alpha X +\beta Y$ for some $\alpha, \beta \in \mathbb R$, it is easy to see from \eqref{primerpaso} that $\alpha$ has to be zero; therefore we have the following system for the other variables.
$$\begin{array}{rcl}
\lambda a_1+ a_2 & = & \beta a_1\\
\lambda a_2 & = & \beta a_2.
\end{array}
$$
If $\beta=0$ from the second equation one gets $a_2=0$, which implies (from the first equation) that also $a_1=0$, contradicting that $Y$ is in a basis that spans the subalgebra. Thus $\beta =\lambda$ and $a_2=0$. With a convenient change of basis, this Lie subalgebra is spanned by $E_1$ and $E_3+a E_2$. 

Thus, a Lie subalgebra of $\mathfrak r_{3,\lambda}$ with $\lambda\neq 0$ is spanned by one of the  sets below
\begin{itemize}
	\item $\{E_1,E_2\}$, 
	\item $\{E_3+a E_2, E_1\}$ with $a\in \mathbb R$,
\end{itemize}
with corresponding Lie subgroups 
$$\mathcal K_0=\{(x,y,0): x,y\in \mathbb R\}, \quad \mathcal K_{a}=\left\{\left(x,\frac{(e^{\lambda z}-1)}{\lambda}a,
 z\right): x,z\in \mathbb R\right\}.$$

Summarizing, we have the following proposition.

\begin{prop} \label{prop11} A Lie subgroup of dimension two on  $S_{3,\lambda}$ for $\lambda\in \mathbb R$,   is one of the following sets
	
	\begin{enumerate}[(i)]
		 \item $\mathcal K_0=\{(x,y,0): x,y\in \mathbb R\}$. This is an abelian subgroup with Lie algebra spanned by $\{E_1, \,E_2\}$. 
	
	\item $\mathcal K_{a}=\left\{\left(x,\frac{(e^{\lambda z}-1)}{\lambda}a
 ,z\right): x,z\in \mathbb R\right\}$ for $\lambda\ne 0$. The corresponding Lie algebra is spanned by
  $\{E_3+aE_2, E_1\}$ for $a\in \mathbb R$ and it  is not abelian but it is solvable.
\item  $\mathcal K_{a}=\{(x,az
 ,z): x,z\in \mathbb R\}$ for $\lambda= 0$. The corresponding Lie algebra is spanned by
  $\{E_3+aE_2, E_1\}$ for $a\in \mathbb R$ and it is abelian.
  \end{enumerate}
\end{prop}

\subsubsection{Subgroups of  $S_{3,\lambda_1,\lambda_2}$.} 
Recall that this group is associated to the Lie algebra described by \eqref{solv2}. We assume that $\lambda_1\neq 0$ (since at least one of the $\lambda_i$ must be non-zero for $\mathfrak{r}_{3,\lambda_1,\lambda_2}$ not to be an abelian Lie algebra). As above, let $\mathfrak a$ be a Lie subalgebra of dimension two. 

As before, if $X$ and $Y$ belong to the subspace spanned by $E_1$ and $E_2$, then $[X,Y]=0$ and $\mathfrak a =\textrm{span }\{E_1,E_2\}$ is an abelian subalgebra. 

Assume that $X\notin \textrm{span }\{E_1,E_2\}$, equivalently $X=E_3+a_1E_1+a_2E_2$, and $Y=b_1E_1+b_2E_2$. In this case one has
$$[X,Y] = \lambda_1 b_1 E_1+ \lambda_2 b_2 E_2.$$
Up to an appropriate choice  of basis, we have in addition $[X,Y]=\beta Y$. Then
$$\begin{array}{rcl}
\lambda_1 b_1& = & \beta b_1\\
\lambda_2 b_2 & = & \beta b_2.
\end{array}
$$

Since $\lambda_1\neq 0$ then either $\beta=\lambda_1$ or $b_1=0$. In the former case, either 
$\lambda_2=\beta=\lambda_1$ (and $b_1, b_2$ can take any value) or $b_2=0$.
In the case that $b_1=0$ we necessarily have that $\lambda_2=\beta$ (because $b_1$ and $b_2$ cannot be 0 simultaneously).

We can summarize the results as follows.



\begin{prop}\label{prop22}  The Lie subgroups of  $S_{3,\lambda_1, \lambda_2}$ of dimension two  are classified as follows.

	\begin{enumerate}[(i)]
		\item For any choice of $\lambda_1$ and $\lambda_2$ we have $\mathcal K_0=\{(x,y,0): x,y\in \mathbb R\}$ is an abelian subgroup with Lie subalgebra 
		$\kk_0=\textrm{span }\{E_1, \,E_2\}$.
		\item For  $\lambda_1\neq \lambda_2$ and $\lambda_2\ne 0$ we also have the subgroups
		\begin{enumerate}[(a)]

			\item  $\mathcal K_{1,b}=\left\{\left(x,\frac{(e^{\lambda_2 z}-1)}{\lambda_2}b,z\right): x,z\in \mathbb R\right\}$, for every  $b\in\mathbb R$. Its associated Lie subalgebra  is $\mathfrak k_{1,b}=\textrm{span }\{E_1,E_3+b E_2\}$ and it is non-abelian;
			
		\item $\mathcal K_{2,a}=\left\{\left(\frac{(e^{\lambda_1 z}-1)}{\lambda_1}a,y,z\right): y,z\in \mathbb R\right\}$ with  $a\in\mathbb R$, which is a non-abelian Lie subgroup. The associated  Lie subalgebra is 
 $\mathfrak k_{2,a}=\textrm{span }\{E_2,E_3+a E_1\}$.

		\end{enumerate}
		
		\item For  $\lambda_2= 0$ (and $\lambda_1\ne 0$) we have
		
		\begin{enumerate}[(a)] 		
		\item 	$\mathcal K_{1,b}=\{(x,bz,z): x, z\in \mathbb R\}$,  with $b\in\mathbb R$. The associated Lie subalgebra is $\mathfrak k_{1,b}=\textrm{span }\{E_1,E_3+b E_2\}$ and it is an  abelian subgroup.
		
		  \item $\mathcal K_{2,a}=\left\{\left(\frac{(e^{\lambda_1 z}-1)}{\lambda_1}a,y,z\right): y,z\in \mathbb R\right\}$ with  $a\in\mathbb R$, which is an abelian Lie subgroup. The associated  Lie subalgebra is  $\mathfrak k_{2,a}=\textrm{span }\{E_2,E_3+a E_1\}$.
		
		\end{enumerate}
		
		\item For $\lambda_1=\lambda_2$ we have 
		\begin{enumerate}[(a)]
			\item  $\mathcal K_{3,c,d}=\left\{\left(x, cx+\frac{(e^{\lambda_1 z}-1)}{\lambda_1}d,z\right): x,z\in \mathbb R\right\}$ for every fixed $c,d\in \mathbb R$. The associated subalgebra  is $\mathfrak k_{3,c,d}=\textrm{span }\{E_1+cE_2,E_3+dE_2\}$ and it is a non-abelian subgroup.

						\item $\mathcal K_{4,e,f}=\left\{\left(ey + \frac{(e^{\lambda_1 z}-1)}{\lambda_1}f , y,z\right): y,z\in \mathbb R\right\}$. The associated 
						 Lie subalgebra is $\mathfrak k_{4,e,f}=\textrm{span }\{E_2+eE_1,E_3+fE_1\}$
						and it is a non-abelian subgroup.

		\end{enumerate}
	\end{enumerate}
\end{prop}

\subsubsection{Subgroups of  $S_{3,\lambda}'$.} In this case, the  corresponding Lie algebra is given by \eqref{solv3}.

Note that as in the previous cases the subspace $\mathfrak k=\textrm{span }\{E_1,E_2\}$ is an abelian subalgebra. 

Assume there exists a subalgebra of dimension two, spanned by vectors $X$ and $Y$, such that $X\notin \textrm{span }\{E_1,E_2\}$. Equivalently, we assume that $X=E_3+a_1 E_1+a_2 E_2$ and $ Y=b_1 E_1+b_2E_2$. Using the definition of this algebra we have
$$[X,Y]=b_1(\lambda E_1-E_2)+b_2(E_1+\lambda E_2)=(\lambda b_1+b_2) E_1 + (-b_1+\lambda b_2) E_2.$$

\medskip

Imposing that  $[X,Y]=\alpha Y$ (where $\alpha$ could be trivial) we have that
 the coefficients satisfy 
 the following system
$$\begin{array}{rcl}
	\alpha b_1 & = & \lambda b_1+b_2\\
	\alpha b_2 & = & -b_1+\lambda b_2,
\end{array}$$
that only has the trivial solution.  We conclude the following classification for 2-dimensional subgroups of 
 $S_{3,\lambda}'$.

\begin{prop}\label{prop33}  The unique Lie subgroup of dimension two of $S_{3,\lambda}'$ is the set $\mathcal K_0=\{(x,y,0)\}$ with Lie algebra $\mathfrak k=\textrm{span}\{E_1,E_2\}$.
\end{prop}

\subsubsection{Subgroups of $\mathrm{SL}(2,\mathbb R)$ and $\mathrm{SO}(3)$.} As above, we study subalgebras in the corresponding Lie algebras $\mathfrak{sl}_2$ and $\mathfrak{so}(3)$. 

We first study
 $\mathfrak{sl}_2$ and  observe that there is no abelian subalgebra of dimension two. Note first that
  the subspace $\mathfrak b=\textrm{span}\{E_2, E_3\}$ is not a subalgebra. Assume  that there is an abelian subalgebra $\kk$ spanned by $X,Y$, where
 $X=E_1+U$, $Y=V$ with $U,V\in \mathfrak b$. Since we assume that the algebra is abelian we would have that  $0=[X,Y]$.  Observing that $[U,V]=\alpha E_1$ and $[E_1,V]\in \mathfrak b$  one obtains $\alpha=0$ (hence $U=\beta V$) and $[E_1,V]=0$, which implies  $V=0$, leading to a contradiction. 
 
 Since in dimension two there are only two non-isomorphic Lie algebras, we conclude that  any subalgebra   $\mathfrak k$ of $\mathfrak{sl}_2$  with dimension two must be solvable. If it is non-abelian, then it  satisfies $\mathfrak k=\textrm{span }\{X,Y\}$ with $[X,Y]=\alpha Y$ and $\alpha \neq 0$.
 Note in addition  that $\mathfrak k\neq\textrm{span }\{E_2, E_3\}$ since $[E_3,E_2]=E_1\notin \textrm{span }\{E_2, E_3\}$.
   
 We assume first that
$X\notin \textrm{span }\{E_2,E_3\}$ and $Y\in \textrm{span }\{E_2,E_3\}$. Equivalently, we can write $X=E_1+a_2E_2+a_3E_3$ and $Y=b_2E_2+b_3E_3$ with $[X,Y]=\alpha Y$, which gives the following system
$$\begin{array}{rcl}
	0 & = & a_3b_2-a_2b_3\\
	\alpha b_2 & = & -2 b_2\\
	\alpha b_3 & = & 2b_3.
\end{array}$$
The second and third equation imply that either $b_2=0$ or $b_3=0$ (note that at least one of them has to be non-zero for $Y$ to be non-trivial). Then, the first equation implies that  $a_2=0$ or $a_3=0$, respectively. Hence the possible bases are $\{ E_1, E_3 \}$ (with $\alpha=2$) and  $\{ E_1, E_2 \}$ (with   $\alpha=-2$).

Now, we consider the case  $Y\notin \textrm{span }\{E_2,E_3\}$ and  $X\in \textrm{span }\{E_2,E_3\}$; thus, $X=a_2E_2+a_3E_3$ and $Y=E_1+b_2E_2+b_3E_3$ with $[X,Y]=\alpha Y$. This gives the following system
$$\begin{array}{rcl}
	\alpha & = & a_3b_2-a_2b_3\\
	\alpha b_2 & = & 2 a_2\\
	\alpha b_3 & = & -2 a_3.
\end{array}$$
The solution to this system gives us the basis $X=\rho E_2+ \frac{1}{\rho}E_3,$  $ Y=E_1+\rho E_2- \frac{1}{\rho}E_3$, that is associated to the
subalgebra  of real matrices of trace zero described by
$$\left\{\left( \begin{matrix}
	t &  \frac{1}{\rho} (s-t)\\
	\rho (s+t) & -t
\end{matrix}\right) \, : \, s,t\in \mathbb R\right\}.$$

An analogous computation shows that there are no subalgebra of dimension two in
 $\mathfrak{so}(3)$: as above  it is possible to verify  that 
there is no 
abelian one and that  subspace spanned by $E_2$ and $E_3$ is not a  subalgebra. 

Assume that there is a 2-dimensional subalgebra which is spanned by elements $X, Y$ such that $X=E_1+a_2E_2+a_3E_3$,  $Y=b_2E_2+b_3E_3$ with $[X,Y]=\alpha Y$. This gives the following system
$$\begin{array}{rcl}
	0 & = & a_3b_2-a_2b_3, 
	\\
	\alpha b_2 & = & -b_3,\\
	\alpha b_3 & = & b_2,
	\end{array}
$$ which has the only the solution $b_2=b_3=0$, that is not admissible (since $Y$ would be 0).
 
Now assume that $X=a_2E_2+a_3E_3, \, Y=E_1+b_2E_2+b_3E_3$ with $[X,Y]=\alpha Y$. In this case we get
 $$\begin{array}{rcl}
 	\alpha & = & -(a_3b_2-a_2b_3), 
	\\
 	\alpha b_2 & = & a_3,\\
 	\alpha b_3 & = & -a_2,
 \end{array}
 $$
which also only has the trivial solution. This proves the next proposition.

\begin{prop}\label{prop23} $\quad$
\begin{itemize}
\item  Lie subgroups of dimension two of the Lie group $\mathrm{SL}(2,\mathbb R)$ are classified as follows.
\begin{enumerate}[(i)]
\item $\mathcal K_1=\left\{  \left( \begin{matrix}
a & b\\
0 & a^{-1} \end{matrix} \right)
\right \}$, for $a\in \mathbb R\setminus\{0\}$,  solvable subgroup with Lie algebra spanned by  $\{E_1, E_3\}$.
\item $\mathcal K_2=\left\{  \left( \begin{matrix}
a & 0\\
b & a^{-1} \end{matrix} \right)
\right \}$ for $a\in \mathbb R\setminus\{0\}$,  solvable subgroup with Lie algebra spanned by $\{E_1, E_2\}$.
\item $\mathcal K_3$ a Lie subgroup with Lie subalgebra spanned by $\{E_2+E_3, E_1-E_3+E_2\}.$
\end{enumerate}
\item The Lie group $\mathrm{SO}(3)$ with Lie algebra $\mathfrak{so}(3)$ does not have Lie subgroups of dimension two.
\end{itemize}

\end{prop}

\section{Solitons among Lie subgroups in dimension three} \label{sec:solitons}

In this section we study soliton  solutions to Mean Curvature Flow that have a subgroup as initial condition.  More precisely, given $K$ as one of the subgroups described in Section \ref{sec:preliminaries} we look for a one-parameter family of isometries  $\psi (\cdot, t)$ of the ambient manifold $G$ such that $\psi(K, 0)$ is an initial embedding of $K$ into $G$ and $\psi (K,t)$ satisfies \eqref{eq:MCF}.\\

 Since $\psi(\cdot, t)$ is an isometry it is well known that $$H(\psi(p, t))= H(\psi(p, 0)).$$
Moreover, since $ \psi(\cdot, t)$ is a one parameter family of isometries we have that $\frac{\partial \psi}{\partial t}(p)=V(p), $ where $V$ is a Killing vector field. Hence, we obtain that  \eqref{eq:MCF} is equivalent to 
\begin{equation}\label{eq-EvolutionKilling1}
	g\left(V, \nu\right)=-H.
\end{equation}
The equivalence between \eqref{eq:MCF} and \eqref{eq-EvolutionKilling1} is well-known in the literature; it appears in several earlier works, see for instance \cite{HS}.\\

From Lemma \ref{lema1} we have that  given  a left-invariant metric on a Lie group $G$  then mean curvature of every Lie subgroup $K$ is constant. Moreover, for unimodular Lie groups, that constant is 0 (or equivalently,  the Lie subgroups are minimal) and their evolutions are trivial (they have speed 0 and in consequence they  remain static in time).  We recall again that the unimodular Lie groups in dimension three are
the Heisenberg Lie group  (given by $S_{3,0}$), $S_{3,\lambda, -\lambda}$(associated to the $Sol$ manifold), $S'_{3,0}, \, \mathrm{SL}(2,\mathbb R)$ and $\mathrm{SO}(3)$. \\

In what follows we concentrate in the study of non-unimodular groups.
We first consider explicit Riemannian metrics on each of these manifolds. To this end, 
on every Lie algebra, we consider a metric such that the basis $\{E_1,E_2,E_3\}$ is orthonormal. On the corresponding Lie group this is equivalent  to a left-invariant metric that imposes that the left-invariant vector fields $E_1, E_2, E_3$ define an orthonormal basis. These vector fields can be described as
$$E_1(x,y,z) :=\frac{d}{ds}\Big|_{s=0} (x,y,z)(s,0,0),$$ $$E_2(x,y,z) :=\frac{d}{ds}\Big|_{s=0} (x,y,z)(0,s,0),$$ 
$$E_3(x,y,z):=\frac{d}{ds}\Big|_{s=0} (x,y,z)(0,0,s).$$
Then, the  desired Riemannian metrics are the following. 
\begin{equation} \label{metrics}
	\begin{array}{ll}
		S_{3,\lambda}:& \qquad g=e^{-2\lambda z}dx^2 + (1+z^2)e^{-2\lambda z}dy^2 + dz^2 - z e^{-2\lambda z}dx dy,\\
		S_{3,\lambda_1,\lambda_2}:& \qquad g=e^{-2\lambda_1z}dx^2 + e^{-2\lambda_2z}dy^2 + dz^2, \\
		S_{3,\lambda}':& \qquad g=e^{-2\lambda z}dx^2 + e^{-2\lambda z}dy^2 + dz^2.
	\end{array}
\end{equation}

Note that the vectors $E_i$ do not coincide in the different groups.

To study the evolutions we will also need to describe the Killing vector fields for each group.
 In our cases of interest, these fields  were  already computed in \cite{HL} and \cite{CR} for  unimodular and non-unimodular  groups, respectively.  Although we will discuss each group separately, in order to describe the Killing vector fields arising from isometries,  it is convenient to  first define a left-translation $L_p$ as

$$L_p x=px,$$
where $p=(p_1, p_2, p_3), \, x=(x_1,x_2, x_3) \in G$ and $px$ refers to $p$ applied to  $x$ (from the left) according to the group operation.

A left-translation $L_{p}$ by an element $p$ of the group is an isometry and
every isometry of  the Lie group $G$ can be written as  
\begin{equation}\label{isomdec}
	f=L_p \circ \phi, \quad\mbox{$\phi$ is an  isometry preserving the identity element, and } p\in G. 
	\end{equation} 

Thus, the group of isometries is essentially determined by the isotropy subgroup (the subgroup of isometries that preserve the identity element).  We shall denote $\Iso(G)=G.K$, which means that every isometry can be written as a composition of isometries as in Equation \eqref{isomdec}.
 
Let $W$ be an element of the Lie algebra of $G$, thus $\exp(tW)\in G$ and it holds
$$\exp(tW)p=R_p\exp(tW)\qquad \mbox{for any }p\in G,$$
where $R_p$ is a right-translation and $\exp:\mathfrak{g}\to G$ is the exponential map in Lie theory (that in general does not agree with the exponential map in Riemannian geometry).
Thus $L_{\exp(tW)}(p)=R_p \exp(tW)$ and the corresponding Killing vector field is given by the right-invariant vector field
$\tilde{W}(p)=\frac{d}{dt}\big|_{t=0}R_p\exp(tW)$. 

In order to write a right-invariant vector field  with respect to a  basis of left-invariant ones we note that
 $$R_p\exp(tW)=L_p \circ L_{p^{-1}}\circ R_p\exp(tW)=L_p(\Ad_{p^{-1}}(tW))\qquad \mbox{for }p\in G.$$
This last equation means that the right-invariant vector field for $W$ at $p\in G$ coincides with left-invariant vector field at $p$ given by $\Ad_{p^{-1}}W$. 

\begin{rem}
A Lie group with isometries given by left-translations becomes a homogeneous space and a solvable Lie group equipped with a left-invariant metric is known as  a  solvmanifold. 
\end{rem}

We conclude this section by showing bases of Killing vector fields for each of our groups and solutions to \eqref{eq-EvolutionKilling1} that arise from the subgroups computed in Section \ref{sec:preliminaries}.  We remark that the Killing vector fields arising from  $E_i$ will be denoted by $\tilde{E}_i$. Since the definition of $E_i$ depends on the underlying group, the corresponding vector fields $\tilde{E}_i$ also depend on the group.

The rest of this section is devoted to prove the following theorem.

\begin{thm} \label{main theorem} 
	
	Let $G$ be a non-unimodular  Lie group of dimension three equipped with a left-invariant metric given by \eqref{metrics}.  Then  $K\leq G$ is a solution to \eqref{eq-EvolutionKilling1} if and only if $K$ is either abelian or it has  mean curvature $H=0$ (in which case  we refer to the translator as {\it trivial or minimal}). All these solutions can be described as follows
	
	
	\begin{enumerate}[(i)]
		\item On $S_{3,\lambda}$: the Lie subgroup $\mathcal K_0$ is a translator in the direction of the non-tangent Killing vector field 
		$$2\lambda [(\lambda x+y) \, \partial_x + \lambda y \, \partial_y+ \partial_z].$$
			
			The Lie subgroup $\mathcal K_a$ is a translator only for $a=0$ and it is minimal. 
		
	\item On $S_{3,\lambda_1,\lambda_2}$ with  $\lambda_1+\lambda_2\neq 0$.

	\begin{itemize}
		\item For $\lambda_1\neq \lambda_2$, with  $\lambda_2\neq 0$,  the subgroups $\mathcal K_0$ defines a non-trivial translator in the direction of the non-tangent Killing vector field 
		$$(\lambda_1+\lambda_2) ( \lambda_1 x \, \partial_x + \lambda_2 y \, \partial_y+ \partial_z).$$
		
		 On the other hand  $\mathcal K_{1,0}$ and $\mathcal K_{2,0}$ are minimal surfaces (hence static solutions).

		\item For $\lambda_2=0$ the non-trivial solutions are given by the subgroup $\mathcal K_0$, that  defines  a translator for a non-tangent Killing vector field 
		$$\lambda_1 ( \lambda_1 x \, \partial_x + \partial_z),$$ 
		 and  
	        the Lie subgroup  $\mathcal K_{1,b}$ that  is a translator for the non-tangent Killing vector field 
	        $$-\lambda_1 b \partial_y.$$ The  Lie subgroup $\mathcal K_{2,0}$ is minimal.
	\item For $\lambda_1=\lambda_2:=\lambda$: the Lie subgroup $\mathcal K_0$ is a translator in non-tangential  direction given by  $$2\lambda ( \lambda x \, \partial_x + \lambda y \, \partial_y+ \partial_z)$$ 
	
	The Lie subgroups $\mathcal K_{3,c,0}$ and $\mathcal K_{4,e,0}$  are minimal. 
	\end{itemize}
\item On $S_{3,\lambda}'$ the Lie subgroup $\mathcal K_0$ is a translator in the (non-tangential) direction $$V=2\lambda [(\lambda x +y)\partial_x + (-x+\lambda y) \, \partial_y+ \partial_z].$$ 
\end{enumerate}

\end{thm}

\begin{rem} In Theorem \ref{main theorem} {\it non-trivial} refers to translators that have non-zero mean curvature. In some cases (e.g. $\mathcal K_{1,b}$ in $S_{3,\lambda_1,0}$), although for most choices of the parameters the solutions  are non-trivial,  for some particular ones this is not the case. For  non-trivial solutions we indicate the direction of translation, which is associated to a non-tangent Killing vector field.

\end{rem}

Note that Theorem \ref{main theorem} implies Theorem \ref{thm:mainintro}, while Theorem \ref{thm:mainintro2} can be deduced from the following result. 

\begin{thm} \label{thm:main2} 
	
 	Let $G$ be a non-unimodular  Lie group of dimension three equipped with a left-invariant metric given by \eqref{metrics}. Let $K$ denote a non-abelian subgroup of $G$ with  non-vanishing mean curvature, then its evolution  is given by one of the following
	
 \begin{itemize}
   
 \item On $S_{3,\lambda}$ we have the solution 
  $$\varphi(x,z,t)=\left(x, \frac{e^{\lambda (z+\lambda t)}-1}{\lambda} a, z\right) \hbox{ with } x,z, a\in \RR. $$

 \smallskip
 
 \item  
   On $S_{3,\lambda_1, \lambda_2}$ we have the following
    solutions

\begin{enumerate}[(a)]

			\item  For  $\lambda_1\neq \lambda_2$ and $\lambda_2\ne 0$ 
			\begin{align*} \varphi_1(x,z,t)&=\left(x,\frac{(e^{\lambda_2 z + (\lambda_1+\lambda_2)\lambda_2t}-1)}{\lambda_2}b,z\right)  \hbox{ with } x,z\in \mathbb R, \\
			\varphi_2(y,z,t)&=\left(\frac{(e^{\lambda_1 z +(\lambda_1+\lambda_2)\lambda_1t}-1)}{\lambda_1}a,y, z\right)  \hbox{ with } y,z\in \mathbb R.
			\end{align*}

		\item For  $\lambda_2= 0$ (and $\lambda_1\ne 0$) we have
		
		\begin{align*} \varphi_1(x,z,t)&=(x, b(z+\lambda_1 t),z),\\\varphi_2(y,z,t)&=\left(\frac{(e^{\lambda_1 z + (\lambda_1+\lambda_2)\lambda_1t}-1)}{\lambda_1}a,y, z\right)  \hbox{ with } y,z\in \mathbb R.
			\end{align*}

		\item For $\lambda_1=\lambda_2:=\lambda$ we have 
		\begin{align*} \varphi_1(x,z,t)&=\left(x, cx+\frac{(e^{\lambda( z+2\lambda t)}-1)}{\lambda}d,z\right)\hbox{ where } x,z\in \mathbb R, \\ 
		\varphi_2 (y,z,t)&=\left(ey + \frac{(e^{\lambda ( z +\lambda t)}-1)}{\lambda}f , y, z\right) \hbox{ where } y,z\in \mathbb R.
		\end{align*}

		\end{enumerate}
		\end{itemize}

\end{thm}

\begin{rem}\label{comparisonwithknown}
The following special cases of Theorem  \ref{main theorem} were already known in the literature:

\begin{itemize}

\item The group $S_{3,\lambda_1,\lambda_2}$ with $\lambda_2=0$ and $\lambda_1=1$  endowed with the suitable metric of \eqref{metrics} corresponds with $\mathbb{H}^2\times \mathbb{R}$ and our solutions are of the type discussed in
 \cite{BuenoLopez2}. In the context of that work, $\mathcal K_0$ is an {\it $h$-translator}, while $\mathcal K_{1,b}$ is a {\it $v$-translator}. 
 
 \noindent Note that to compare the model used in \cite{BuenoLopez2} with ours is necessary to consider the change of coordinates $\tilde{z}=e^{-z}$, then $e^{-2z} dx^2+ dz^2=\frac{dx^2 +d\tilde{z}^2}{\tilde{z}^2}$ is the metric on $\mathbb{H}^2$. In consequence
our $y$ corresponds with the variable $z$ in \cite{BuenoLopez2}.

\item The group $S_{3,\lambda_1, \lambda_2}$ with $\lambda_1=\lambda_2=1$ endowed with the suitable metric of \eqref{metrics} is the hyperbolic 3-space and the evolution of $\mathcal{K}_0$ described by Theorem \ref{main theorem} correspond to horo-spheres. These solutions were previously known in the literature  (see \cite{BuenoLopez1} and  references therein). 
 
 \noindent The half-space model of $\mathbb{H}^3$ can be obtained by considering  $\tilde{z}=e^{-z}$ and in consequence $e^{-2z} dx^2+e^{-2z}dy^2+ dz^2=\frac{dx^2 +dy^2+d\tilde{z}^2}{\tilde{z}^2}$ 

\end{itemize}

\end{rem}

In the forthcoming subsections we compute in each case the conditions under which the subgroups obtained in Section \ref{sec:subgroups} define a translator.

\subsection{The Lie group $S_{3,\lambda}$} \label{groups3l}

Since $S_{3,0}$ is unimodular we assume $\lambda\ne0$.
A basis of left-invariant vector fields of $S_{3,\lambda}$ is given by
$$E_1(x,y,z)=e^{\lambda z}\partial_x, \, E_2(x,y,z)=z e^{\lambda z}\partial_x+e^{\lambda z}\partial_y, \, E_3(x,y,z) = \partial_z, $$
satisfying the  non-trivial Lie bracket relations:
$$[E_3,E_1]=\lambda E_1, \qquad [E_3,E_2]=E_1+\lambda E_2.$$

 In $S_{3,\lambda}$, translations on the left are generated by 
 $$\begin{array}{c}
	(x,y,z) \to (x+c, y, z),  \\
	(x,y,z) \to (x, y+c, z),  \\ 
	(x,y,z) \to (e^{\lambda c}x+ c e^{\lambda c}y, e^{\lambda c}y, z+c),
\end{array}
$$
where $c\in \mathbb{R}$.  The Killing vector fields arising from $E_i$, $i=1,2,3$, are
\begin{align*}
	\widetilde{E}_1(x,y,z) & = \frac{d}{ds}\Big|_{s=0}(s,0,0)(x,y,z)=\partial_x=e^{-\lambda z}E_1,\\
	\widetilde{E}_2(x,y,z) & =  \partial_y= e^{-\lambda z}(E_2-z E_1),\\
	\widetilde{E}_3(x,y,z) & =  (\lambda x+y) \, \partial_x + \lambda y \, \partial_y+ \partial_z= e^{-\lambda z}(\lambda x -\lambda y z+y)E_1 + \lambda y e^{-\lambda z}E_2+ E_3.
\end{align*}

For $\lambda \neq 0$, the isometry group has dimension three and corresponds to the group of translations on the left.  Thus any  Killing vector field in $S_{3,\lambda}$ can be written as \begin{equation}V=\eta \tilde{E}_1 + \beta \tilde{E}_2+ \mu \tilde{E}_3 \hbox{ for }\eta, \beta, \mu \in \RR. \label{generickilligs3lambda}\end{equation}  
 
Now we study the evolutions by isometries of the subugroups described in Section \ref{subgroupss3lambda} .

\subsubsection{ The subgroup $\mathcal K_0$ }
	
	Recall that $\mathcal K_0$ is associated to the Lie algebra $\kk=\textrm{span }\{E_1, E_2\}$ and its normal
 vector is $\nu=  E_3$. Its trace is given by $tr(\ad_{\nu})= 2\lambda$. Thus its mean curvature is $$H=-2\lambda.$$
From Equation \eqref{eq-EvolutionKilling1} and Equation \eqref{generickilligs3lambda}
one gets 	$$2\lambda=\mu, $$
	which gives $\mathcal K_0$ is a translator for the Killing vector  $V=\eta \tilde{E}_1 + \beta \tilde{E}_2+ 2\lambda  \tilde{E}_3$. Since $\tilde{E}_1$ and $\tilde{E}_2$ are tangent vectors, the constants $\eta$ and $\beta$ can be taken 0 by reparametrizing the solution.
	
\subsubsection{ The subgroup $\mathcal K_a$ }	
	
The Lie subgroup $\mathcal K_a$ is associated to the Lie algebra  \\
$\kk_a=\textrm{span }\left\{\frac{1}{\sqrt{1+a^2}}(aE_2+E_3),E_1\right\}$. The unit normal is given by  $\nu=\frac{1}{\sqrt{1+a^2}}(aE_3-E_2)$ and the trace of $\ad_{\nu}$ is  $tr(\ad_{\nu})=\frac{2a\lambda}{\sqrt{1+a^2}}$.  Thus the mean curvature of  $\mathcal K_a$ is
$$H=-\frac{2a\lambda}{\sqrt{1+a^2}}.$$
Imposing Equation\eqref{eq-EvolutionKilling1} for $V$ given by  \eqref{generickilligs3lambda} one obtains 	$$ 2a\lambda +a \mu= (\beta+\lambda \mu y)e^{-\lambda z}.
$$ 
Since  the left-hand side is constant, while the right hand-side is a function of $y$ and $z$, we necessarily have that $\beta=\mu=0$. This implies that $a=0$ and $\mathcal K_0$ is a minimal surface.

\subsection{ The group  $S_{3,\lambda_1,\lambda_2}$} \label{groups3l1l2}
	
Recall that for $\lambda_1+\lambda_2=0$  the group  $S_{3,\lambda_1,\lambda_2}$ is unimodular, hence throughout this section we assume that $\lambda_1+\lambda_2\ne0$.
	
A basis of left-invariant vector fields in  $S_{3,\lambda_1,\lambda_2}$ is given by
		$$E_1(x,y,z)=e^{\lambda_1 z}\partial_x,\,  E_2(x,y,z)=e^{\lambda_2z}\partial_y, \, E_3(x,y,z) = \partial_z, $$
which satisfy the  non-trivial Lie bracket relations:
	$$[E_3,E_1]=\lambda_1 E_1, \qquad [E_3,E_2]=\lambda_2 E_2.$$
	
Translations on the left by elements of the groups are generated by the following 
	one-parameter groups 	
	$$\begin{array}{c}
		(x,y,z) \to (x+c, y, z),  \\
		(x,y,z) \to (x, y+c, z),  \\ 
		(x,y,z) \to (e^{\lambda_1 c}x, e^{\lambda_2c}y, z+c),
	\end{array}
	$$
	where $c\in \mathbb{R}$.  The Killing vector fields arising from $E_i$, $i=1,2,3$, are
	\begin{align*}
		\widetilde{E}_1(x,y,z) & = \frac{d}{ds}\Big|_{s=0}(s,0,0)(x,y,z)=\partial_x=e^{-\lambda_1 z}E_1,\\
		\widetilde{E}_2(x,y,z) & =  \partial_y= e^{-\lambda_2 z}E_2,\\
		\widetilde{E}_3(x,y,z) & =  \lambda_1 x \, \partial_x + \lambda_2 y \, \partial_y+ \partial_z= \lambda_1 x e^{-\lambda_1 z}E_1 + \lambda_2 y e^{-\lambda_2 z}E_2+ E_3.
	\end{align*}
	Recall that although in the previous subsection  the  Killing vectors we also denoted by $\widetilde{E}_i$, they do not agree with the ones above.\\

	Denote by $L_a^1, L_b^2, L_c^3$ the maps associated to the previous Killing fields. More precisely, let
	$$L_a^1(x,y,z) = (x+a, y, z), \quad L_b^2	(x,y,z) = (x, y+b, z), \quad L_c^3(x,y,z) = (e^{\lambda_1 c}x, e^{\lambda_2c}y, z+c).$$
	
	Since $L: G\to \Iso(G)$ is a homomorphism of groups,  we have the following relations.	
	\begin{equation}\label{relationsIso}
		L_a^1\circ L_b^2=L_b^2\circ L_a^1, L_c^3 \circ L_a^1=L_{(e^{\lambda_1c}a,0,0)}\circ L_c^3, \mbox{ and } L_c^3 \circ L_b^2=L_{(0,e^{\lambda_2c}b,0)}\circ L_c^3. 
	\end{equation}
	
	In $S_{3,\lambda_1,\lambda_2}$ 
	  there is another isometry that fixes the identity element and it is given by  reversing orientation:
	$$(x,y,z) \to (-x,y,z).$$
	When 
	 $\lambda_1\neq \lambda_2$ and $\lambda_2\neq 0$,  the connected component of the identity of the isometry group is completely described by  left-translations:  $\Iso(S_{3,\lambda_1,\lambda_2},g)=S_{3,\lambda_1,\lambda_2}$, for $\lambda_1\neq\lambda_2$ and $\lambda_2\neq 0$.
 On the other hand,
  if $\lambda_2=0$ (and $\lambda_1\neq 0$), one has  in addition the following monoparametric group of isometries:
 $$T_a(x,y,z)=\left(\begin{matrix}
 	1 & 0 & 0\\
 	0 & 	\cos(au) & -\sin(au) \\
 	0 & 	\sin(au) & \cos(au)
 \end{matrix} \right)\left(\begin{matrix}
 	x\\y\\z \end{matrix} \right).$$
 This family is associated to the  one-parameter family of  Killing vector fields given by  $$\widetilde{T}(x,y,z)=-a z \partial_y + a y \partial_z= a(-zE_2+yE_3),$$ 
 for any $a \in \RR$. In particular, if $\lambda_2=0$ 
 the isometry group has dimension four and its connected component is given by $\Iso_0(S_{3,\lambda,0},g)=S_{3,\lambda_1, 0}.\mathrm{SO}(2)=\mathrm{PSL}(2,\mathbb R) \times \mathbb R$. 

%

 For $\lambda_1=\lambda_2:=\lambda$, the corresponding Lie group  endowed with $g$ is isometric to the hyperbolic space  of sectional curvature  $-\lambda^2$,  the isometry group is known and the connected component of the identity is given by $\mathrm{SO}(3,1)$. 

Summarizing , for $\lambda_1+\lambda_2\neq 0$ the group $\Iso_0(S_{3,\lambda_1, \lambda_2},g)$ is given by 
\begin{itemize}
	\item  $S_{3,\lambda_1, \lambda_2}$ if $\lambda_1\neq\lambda_2$ and $\lambda_2\neq 0$.
	\item $S_{3,\lambda_1, 0}.\mathrm{SO}(2)$ if  $\lambda_2=0$. 
	\item$S_{3,\lambda, \lambda}.\mathrm{SO}(3)=\mathrm{SO}(3,1)$ if $\lambda_1=\lambda_2:=\lambda$.   \end{itemize}
Now, we proceed to study solutions to \eqref{eq-EvolutionKilling1} that are associated the subgroups described in Proposition~\ref{prop22}.

\subsubsection{The case $\lambda_1\ne \lambda_2$ and $\lambda_2\ne 0$. }

In this case, the isometry group consists only of translations on the left and any Killing vector field $V$ can be written as $V=\eta \tilde{E}_1 + \beta \tilde{E}_2+ \mu \tilde{E}_3$.  Recall that we assume $\lambda_1+\lambda_2\ne 0$ (i.e.  $S_{3,\lambda_1, \lambda_2}$ is  not unimodular.)
		
\begin{enumerate}[(i)]
	\item Consider the group $\mathcal K_0$ associated to the Lie algebra $\kk=\textrm{span }\{E_1, E_2\}$. The unit normal is given by $\nu=E_3$ and the trace of $\ad_{\nu}$ equals $\lambda_1+\lambda_2$. Thus 
	$$H=-(\lambda_1+\lambda_2).$$
	Then,  Equation \eqref{eq-EvolutionKilling1} gives: 
	$$g\left(V, E_3\right)= \lambda_1+\lambda_2 \quad \Longrightarrow \quad \mu=\lambda_1+\lambda_2.$$
	Since $\tilde{E}_1$ and $\tilde{E}_2$ are tangential, up to reparametrization $\mathcal K_0$ is a translator in the direction $V=(\lambda_1+\lambda_2) \tilde{E}_3$.
	
	\item  The Lie subgroup $\mathcal K_{1,b}$ with Lie subalgebra $\kk_{1,b}=\textrm{span }\{E_3+bE_2, E_1\}$ for $b\in \RR$ has a unit normal given by
	   $\nu=\frac{1}{\sqrt{1+b^2}}(bE_3-E_2)$. The trace of $\ad_{\nu}$ equals $(\lambda_1+\lambda_2)\frac{b}{\sqrt{1+b^2}}$, or equivalently
	$$H=-(\lambda_1+\lambda_2)\frac{b}{\sqrt{1+b^2}}.$$ 
	 Combining this computation with Equation \eqref{eq-EvolutionKilling1} we have
	$$g\left(V, -\frac{1}{\sqrt{1+b^2}}(E_2-b E_3)\right)=b\frac{\lambda_1+\lambda_2}{\sqrt{1+b^2}},$$
	which  gives 
	$-e^{-\lambda_2 z}\beta -\mu (\lambda_2 y e^{-\lambda_2 z}-b) =b(\lambda_1+\lambda_2)$, or equivalently
	$$ -e^{-\lambda_2 z}(\beta + \mu \lambda_2 y)=b(\lambda_1+\lambda_2-\mu).$$ 
	Since the left-hand side is a function of $y$ and $z$ and the right-hand side is constant, we must have $\beta=\mu=b(\lambda_1+\lambda_2-\mu)=0$. 
	This implies that 
	 the Lie subgroup $\mathcal K_{1,b}$ is a translator if and only if $b=0$, which is a minimal surface.

\item For the Lie subgroup $\mathcal K_{2,a}$ with Lie subalgebra $\kk_{2,a}=\textrm{span }\{E_3+aE_1, E_2\}$,  $a\in \RR$ the unit normal vector is   $\nu=\frac{1}{\sqrt{1+a^2}}(aE_3-E_1)$. Thus the mean curvature is 
$$H=-\textrm{tr} (\ad_{\nu})= -(\lambda_1+\lambda_2)\frac{a}{\sqrt{1+a^2}}.$$
As in the case above, we conclude by
 taking a Killing vector field $V=\eta \tilde{E}_1+\beta\tilde{E}_2+\mu\tilde{E}_3$ that $a=0$ and $\mathcal K_{2,0}$ is minimal. 
 
		
\end{enumerate}

	\smallskip
	
	\subsubsection{ The case $\lambda_2=0$} In this situation the isometry group has dimension four. In particular, every Killing vector field $V$ has the form $V=\eta \tilde{E}_1 + \beta \tilde{E}_2+ \mu \tilde{E}_3 + \rho \tilde{T}$, for $\tilde{T}(x,y,z)=-zE_2+yE_3$.  
	
	\begin{enumerate}
		\item For $\mathcal K_0$ with Lie algebra $\kk_0=\textrm{span }\{E_1,E_2\}$  the unit normal vector is $\nu=E_3$ and  the mean curvature 
		$$H=-\lambda_1.$$
		Equation \eqref{eq-EvolutionKilling1} is equivalent to
		$$\lambda_1= \mu +y \rho.$$
		This implies that necessarily $\rho=0$.
		Since any vector of the form $\eta\tilde{E}_1+\beta \tilde{E}_2$ is tangential, $\mathcal K_0$ is a translator in the direction $V=\lambda_1 \tilde{E}_3$. 
		\item For the Lie subgroup $\mathcal K_{1,b}$ with Lie subalgebra $\kk_{1,b}=\textrm{span }\{E_1, E_3+bE_2\}$ the unit normal vector is $\nu=-\frac{1}{\sqrt{1+b^2}}(E_2-bE_3)$. Thus the trace of $\ad_{\nu}$ is $\lambda_1\frac{b}{\sqrt{1+b^2}}$, or equivalently
			$$H=-\lambda_1\frac{b}{\sqrt{1+b^2}}.$$ 
			Equation \eqref{eq-EvolutionKilling1} implies
			 $$-\beta +b \mu +\rho(z +by)=b\lambda_1.$$
			This equation implies that $\rho=0$. Recalling that $b E_2+E_3$ is tangential, we		
	 conclude for any $b$ that $\mathcal K_{1,b}$ defines a non-trivial translator in the direction  $V= -\lambda_1 b \tilde{E}_2$.
			
			\item For the Lie subgroup $\mathcal K_{2,a}$ with Lie subalgebra $\kk_{1,b}=\textrm{span }\{E_2, E_3+aE_1\}$ the unit vector field is $\nu=\frac{1}{\sqrt{1+a^2}}(aE_3-E_1)$ and
			$$H=-\lambda_1\frac{a}{\sqrt{1+a^2}}.$$ 
			Equation \eqref{eq-EvolutionKilling1} implies
			 $$-e^{-\lambda_1 z}(\eta+\mu\lambda_1 x) +a(\mu+y \rho)=a\lambda_1.$$
			For this equation to hold for every $x,\, y$ and $z$ necessarily $\eta=\mu=\rho=a(\lambda_1-\mu)=0$. This implies $a=0$ and $\mathcal K_{2,0}$ is minimal.
							\end{enumerate}

	\subsubsection{ The case $\lambda_1=\lambda_2:=\lambda$} In this case we need  a more detailed description of the isometries. Observe that 
	 one has the skew-symmetric derivation $D(E_1)=a E_2, D(E_2)=-a E_1, D(E_3)=0$ for any $a \in \RR$.  Thus,  for each $a$ one obtains the following monoparametric group of orthogonal automorphisms of  $S$:
	$$T_u(x,y,z)=\left(\begin{matrix}
		\cos(au) & -\sin(au) & 0\\
		\sin(au) & \cos(au) & 0 \\
		0 & 0 & 1
	\end{matrix} \right)\left(\begin{matrix}
		x\\y\\z \end{matrix} \right).$$
		The corresponding Killing vector field is spanned by:
	$$\widetilde{T}_1(x,y,z)=- y \partial_x +  x \partial_y = -y e^{-\lambda z}E_1+xe^{-\lambda z}E_2.$$
	Thus, together with   the relationships in \eqref{relationsIso}, one has
	$$T_u\circ L_a^1=L_{T_u(a,0,0)}\circ T_u, \quad T_u\circ L_b^2=L_{T_u(0,b,0)}\circ T_u, \quad T_u\circ L_c^3=L_c^3\circ T_u.$$
	
	According to results in \cite{AOPS} there are two additional  linearly independent vector fields given by
	$$\tilde{T}_2(x,y,z)=2x E_1- \lambda x^2 e^{\lambda z} E_3, \quad \tilde{T}_3(x,y,z)=2y E_2- \lambda y^2 e^{\lambda z} E_3.$$
	Hence a general Killing vector field can be written as $V=\eta\tilde{E}_1 + \beta \tilde{E}_2 + \mu \tilde{E}_3 + \rho \tilde{T}_1  + \theta \tilde{T}_2 + \kappa \tilde{T}_3$. 
	
	Now we proceed to compute the possible evolutions.
	
\begin{enumerate}[(i)]
\item  For $\mathcal K_0$ with Lie algebra $\kk_0=\textrm{span }\{E_1,E_2\}$ as above $\nu=E_3$ and the mean curvature is
$$H=-2\lambda.$$
Thus, Equation \eqref{eq-EvolutionKilling1} gives
$$\mu-\lambda e^{\lambda z}(\theta x^2+ \kappa y^2)= 2\lambda$$
from which we deduce that $\mu=2\lambda$ and $\theta=\kappa=0$. Since $\tilde{T}_1$ is tangencial, we have that $\mathcal K_0$ is a translator with respect to  the isometry associated to the Killing vector field $V= 2\lambda  \tilde{E}_3$. 
\item  Take the Lie subgroup  $\mathcal K_{3, c,d}$ with Lie subalgebra  $\kk_{3,c,d}=\textrm{span }\{E_1+cE_2,E_3+dE_2\}$ with $c,d\in \RR$. The unit normal vector is  given by  $\nu=\frac{1}{\sqrt{1+c^2+d^2}}(cE_1-E_2+dE_3)$ and the trace of $\ad_{\nu}$ equals $2\lambda\frac{d}{\sqrt{1+c^2+d^2}}$. Thus, from Equation \eqref{eq-EvolutionKilling1} we obtain
$$\begin{array}{rcl}
			2\lambda d & = & ce^{-\lambda z}(\eta+\mu \lambda x -\rho y)+ 2\theta x c - e^{-\lambda z}  (\beta +\mu \lambda y + \rho x) - 2\kappa y+\\
			&  & + d  e^{-\lambda z}(-\kappa y^2 - \theta x^2) + d \mu.
		\end{array}
$$	
Following a reasoning analogous to previous cases we conclude that $d=0$ and the surface is minimal.

\item  For $\mathcal K_{4,e,f}$ with Lie algebra $\kk_{4,e,f}=\textrm{span }\{E_2+eE_1, E_3+fE_1\}$ with $c,d\in \RR$ we the unit normal given by    $\nu=\frac{1}{\sqrt{1+e^2+f^2}}(eE_2-E_1+fE_3)$. The trace of $\ad_{\nu}$ (which is related to the mean curvature through Lemma \ref{lema1}) is  $2\lambda\frac{f}{\sqrt{1+e^2+f^2}}$. 
			
As in the previous case, the equations imply that $f=0$ and $\mathcal K_{4,e,0}$ is minimal.
\end{enumerate}		
		\subsection {The Lie group $S'_{3,\lambda}$ } \label{groups3prima}
		A basis of left-invariant vector fields is given by:
		$$E_1(x,y,z)=e^{\lambda z}[\cos(z)\partial_x-\sin(z)\partial_y],\,  E_2(x,y,z)=e^{\lambda z}[\sin(z)\partial_x +\cos(z)\partial_y], \, E_3(x,y,z) = \partial_z. $$
	They satisfy the following non-trivial Lie bracket relations:
		$$[E_3,E_1]=\lambda E_1-E_2, \qquad [E_3,E_2]=E_1+\lambda E_2.$$
		
		In this group a basis of right-invariant vector fields written with respect to left-translations is
		\begin{align*}
			\widetilde{E}_1(x,y,z) & = \frac{d}{ds}\Big|_{s=0}(s,0,0)(x,y,z)=\partial_x=e^{-\lambda z}[\cos(z)E_1+\sin(z)E_2],\\
			\widetilde{E}_2(x,y,z) & =  \partial_y= e^{-\lambda z}[-\sin(z)E_1+\cos(z)E_2],\\
			\widetilde{E}_3(x,y,z)  &=  (\lambda x +y)\partial_x + (-x+\lambda y) \, \partial_y+ \partial_z,\quad \mbox{ and} \quad \partial_z=E_3.		\end{align*}

		Recall that although these  Killing vectors use the same notation as in previous subsections, they do not necessarily  agree.
			
		 In addition, for each $a\in \RR$ one has the following monoparametric group of orthogonal automorphisms of  $S'_{3,\lambda}$
			$$T_u(x,y,z)=\left(\begin{matrix}
				\cos(au) & \sin(au) & 0\\
				-\sin(au) & \cos(au) & 0 \\
				0 & 0 & 1
			\end{matrix} \right)\left(\begin{matrix}
				x\\y\\z \end{matrix} \right).$$
			The corresponding Killing vector field is
			$$\widetilde{T}(x,y,z)=a (y \partial_x -  x \partial_y).$$
		Then, a general vector field can be written as $$V=\eta \tilde{E}_1+\beta \tilde{E}_2+ \mu \tilde{E}_3 + \rho \tilde{T}.$$
			
						Recall from Proposition \ref{prop33} that the only Lie subgroup in this case is related to 
	 the abelian subalgebra $\kk=\textrm{span }\{E_1,E_2\}$. Then, $\nu=E_3$ and the mean curvature is
			$$H=-2\lambda.$$ 
			
			Hence,  Equation \eqref{eq-EvolutionKilling1} together with $g(\tilde{E}_1, \nu)=g(\tilde{E}_2,\nu)=g(\tilde{T},\nu) = 0$ gives
			$$\mu=2\lambda.$$
			Thus, $\mathcal K_0$ is a translator with isometries associated to $V=2\lambda \tilde{E}_3$. Note that directions   $\tilde{E}_1$, $\tilde{E}_2$  and $\tilde{T}$
			are tangential, hence they can be assumed 0 after reparametrizing.



 \section{The evolution of subgroups} \label{sec:solutions} 
In this section we conclude the proofs of Theorem \ref{main theorem} and Theorem \ref{thm:main2} by explicitly writing the evolving subgroups that are not minimal. We also include explicit  solutions in the case that the evolutions are not self-similar
 

 
  \subsection{Translating solutions}
 \smallskip

 From our previous computations we have that our self-similar solutions are defined by translations, which for instance in the case of $S_{3,\lambda}$  implies that the evolution is given by
 $$\varphi(\tilde{x},\tilde{y},t)=(e^{2\lambda^2t}\tilde{x},e^{2\lambda^2t}\tilde y, 2\lambda t),  \hbox{ with } \tilde x,\, \tilde y\in \RR. $$
 After the (time-dependent) spatial re-parametrization given by $x= e^{2\lambda^2t}\tilde{x},\, y=e^{2\lambda^2t}\tilde y$ this is equivalent to 
  $$\varphi(x,z,t)=(x,y, 2\lambda t),  \hbox{ with } x,y\in \RR. $$
  
 We write the evolutions in this simple (re-parametrized) form for every group under consideration. For the sake of having a complete list below we repeat the equation above, obtaining the following self-similar solutions to Mean Curvature Flow:
 
  \begin{itemize}
   
 \item On $S_{3,\lambda}$ we have the solution 
  $$\varphi(x,z,t)=(x,y, 2\lambda t),  \hbox{ with } x,y\in \RR. $$

 \smallskip
 
 \item  
   On $S_{3,\lambda_1, \lambda_2}$ we have the solutions
 \begin{align*}   
 \varphi_1(x,y,t)&= ( x, y, (\lambda_1+\lambda_2)t )  \hbox{ with } x,y\in \RR, \, \hbox{ and }\lambda_2\neq 0, \\
 \varphi_2(x,y,t)&= (x, y, \lambda_1t )\hbox{ with } x,y\in \RR, \,  \hbox{ and }\lambda_2= 0,\\
  \varphi_3(x,z,t)&= (x, bz-b\lambda_1 t, z)\hbox{ with } x,y, b\in \RR, \,  \hbox{ and } \lambda_2= 0.
 \end{align*}



 \item  On $S'_{3,\lambda}$.
 $$\varphi(x,y)= (x,y, 2\lambda t) .$$
 
 
 \end{itemize}
\subsection{Non-translating solutions}
Inspired by the previous computations, we can deduce evolutions for  non-commutative subgroups  that are not given by translations. 
To this end we propose an ansatz of solution that has constant mean curvature in space for each fixed time (but it depends on time  $t$).
 More precisely, we construct the proposed solutions as follows:
first observe that our computations in Section \ref{sec:subgroups} show that for subgroups that are not abelian the mean curvature depends on a parameter that defines the group and here we call it generically $A$. In all our examples,  the parametrization of the group includes this parameter $A$ 
multiplied by a factor of the form $\frac{e^{\lambda z}-1}{\lambda}$ (here $\lambda$ refers generically to a parameter that may be $\lambda$, $\lambda_1$ or $\lambda_2$).  In terms of the tangent vectors (that are in the Lie algebra) this is equivalent to a factor of the form $A e^{\lambda z}$. Then we propose a solution that modifies this factor in the Lie algebra by a function of time, namely
by $A c(t) e^{\lambda z}$, where $c(0)=1$. 
 In the parametrization of the
evolution this is equivalent to a term of the form  $\frac{c(t)  e^{\lambda z}-1}{\lambda} A$. A direct computation with this ansatz gives a simple ODE for $c$  that  allow us to obtain  the list of solutions of Theorem \ref{thm:main2}, which  can be verified by explicitly computing their evolutions. For completeness we include one of these verifications in detail and since the other ones are analogous, are left to the interested reader.

\subsubsection{The evolution of $\mathcal K_{1,b}$ in $S_{3,\lambda_1, \lambda_2}$ }
In the following, we compute the evolution under the MCF of the subgroup $\mathcal K_{1,b}$ (see Proposition \ref{prop22} (ii) (a) for $\lambda_2\neq 0$). In fact, consider the one-parameter family of submanifolds with respect to $t$ given by $$\mathcal K^t_{1,b}=\left\{\varphi(x,z,t)=\left(x,\frac{(e^{\lambda_2 z + (\lambda_1+\lambda_2)\lambda_2t}-1)}{\lambda_2}b,z\right): x,z\in \mathbb R\right\}.$$ Here, for every $t$, the tangent space of $\mathcal K^t_{1,b}$ is generated by
\begin{align*}
\tau_1^t=& \partial_x=e^{-\lambda_1 z} E_1,\\
\tau_2^t=&b e^{\lambda_2 z +(\lambda_1+\lambda_2)\lambda_2t} \partial_y+\partial_z = b  e^{ (\lambda_1+\lambda_2)\lambda_2t}E_2+E_3.
\end{align*}
Then the unit normal associated to $\mathcal K^t_{1,b}$ can be written as $\nu^t=\frac{be^{(\lambda_1+\lambda_2)\lambda_2t}E_3-E_2}{\sqrt{1+b^2 e^{-2(\lambda_1+\lambda_2)\lambda_2t}}}$.

The first fundamental form   $\bar{g} $ is determined by $E=\bar{g}(\tau_1, \tau_1)=e^{-2 \lambda_1 z}$, $F=\bar{g}(\tau_1, \tau_2)=0$ and $G=\bar{g}(\tau_2, \tau_2)=1+b^2 e^{2(\lambda_1+\lambda_2)\lambda_2t}$. 

In addition, considering the computations of the connection given in \cite[Section 3]{AOPS} we have that
\begin{align*}
\nabla_{\tau_1^t} \tau_1^t &=  \nabla_{e^{-\lambda_1 z} E_1} (e^{-\lambda_1 z} E_1) = e^{-2\lambda_1 z} \; \nabla_{E_1} E_1 = \lambda_1 e^{-2\lambda_1 z} E_3,\\
\nabla_{\tau_2^t} \tau_2^t &=  \nabla_{b  e^{ (\lambda_1+\lambda_2)\lambda_2t}E_2+E_3} (b  e^{(\lambda_1+\lambda_2)\lambda_2t}E_2+E_3) \\
 &= b  e^{ (\lambda_1+\lambda_2)\lambda_2t} \nabla_{E_2} (b  e^{ (\lambda_1+\lambda_2)\lambda_2t}E_2+E_3) + \nabla_{E_3} (b  e^{(\lambda_1+\lambda_2)\lambda_2t}E_2+E_3) \\
 &=  b^2  e^{ 2(\lambda_1+\lambda_2)\lambda_2t} \nabla_{E_2} E_2 + b  e^{(\lambda_1+\lambda_2)\lambda_2t} \nabla_{E_2} E_3 + b  e^{ (\lambda_1+\lambda_2)\lambda_2t} \nabla_{E_3} E_2 + \nabla_{E_3} E_3 \\
 &= \lambda_2 b^2  e^{2(\lambda_1+\lambda_2)\lambda_2t} E_3 - \lambda_2 b  e^{(\lambda_1+\lambda_2)\lambda_2t} E_2.
\end{align*} Following \cite[Section 5.1.1]{AOPS} and using that $F=0$, we conclude that the mean curvature at time $t$ is constant. Denoting it as $H^t$ we have

$$-H^t =\frac{1}{EG}\left[G g(\nabla_{\tau_1^t}\tau_1^t,\nu^t) + E g(\nabla_{\tau_2^t}\tau_2^t,\nu^t)\right] =\frac{b (\lambda_1+\lambda_2) e^{ (\lambda_1+\lambda_2)\lambda_2t}}{\sqrt{1+b^2 e^{2(\lambda_1+\lambda_2)\lambda_2t}}}.$$

On the other hand,  
\[
\frac{\partial}{\partial t} \varphi= b (\lambda_1+\lambda_2) e^{\lambda_2 z + (\lambda_1+\lambda_2)\lambda_2t} \partial_y =  b (\lambda_1+\lambda_2) e^{(\lambda_1+\lambda_2)\lambda_2t}	E_2,
\] and therefore, $g\left(\frac{\partial}{\partial t} \varphi ,\nu^t\right)=-H^t$.\\

\section{A counterexample in dimension 4 and the proof  of Theorem \ref{thm:mainintro3}}\label{counterex}

  We start this section by showing an example of a Lie group $G$ of dimension 4 and a non-abelian Lie subgroup $K$ of dimension 3 that evolves self-similarly. 

Consider the Lie group $G$ given by  the semidirect product of $\RR$ and the Heisenberg Lie group $S_{3,0}$ by the automorphism $\tau_{s/2}$ of $S_{3,0}$ defined as
$$\tau_{\frac{s}2}(x,y,z)=(e^{\frac{s}2} x, e^{\frac{s}2}y, e^{s}z).$$
Considering the differentiable structure of $\RR^4$, the product on $G$ can be written as
$$(s_1,x_1,y_1,z_1)(s_2,x_2,y_2,z_2)=\left(s_1+s_2, x_1+e^{\frac{s_1}2}x_2, y_1+ e^{\frac{s_1}2}y_2, z_1+e^{s_1}z_2+e^{\frac{s_1}2}\frac{(x_1y_2-x_2y_1)}2\right).$$
We have the following basis of left-invariant vector fields.
$$E_0(s,x,y,z)=\partial_s, \qquad E_1(s,x,y,z)=e^{s/2}\left(\partial_x-\frac{y}2\partial_z\right),$$  $$E_2(s,x,y,z)=e^{s/2}\left(\partial_y+\frac{x}2\partial_z\right), \qquad E_3(s,x,y,z)=e^s\partial_z,$$
which satisfy the non-trivial Lie bracket relations
$$[E_0,E_1]=\frac12E_1\quad [E_0,E_2]=\frac12E_2 \quad [E_0,E_3]=E_3\quad [E_1,E_2]=E_3.$$
On the other hand, right-invariant vector fields can be generated by the following basis.
\begin{align*}
\tilde{E}_0(s,x,y,z)&=\partial_s+\frac{x}2 \partial_x+\frac{y}2\partial_y+z\partial_z =E_0+\frac{e^{-s/2}}2(x E_1+yE_2) + ze^{-s}E_3,\\
\tilde{E}_1(s,x,y,z)&=e^{-s/2}E_1+\frac{y}2 e^{-s}E_3, \qquad \tilde{E}_2(s,x,y,z)=e^{-s/2}E_2-\frac{x}2e^{-s}E_3,\\
  \tilde{E}_3(s,x,y,z)&=e^{-s} E_3.\end{align*}
 
 Consider the subgroup $\mathcal K$ that is isomorphic to the Heisenberg Lie group. Namely,
 $$\mathcal K =\{(0,x,y,z) : x,y,z\in\RR\}.$$
 This is a (not abelian) 2-step nilpotent Lie subgroup, whose   Lie algebra of $\mathcal K$ is spanned by the vectors $E_1,E_2, E_3$.
 
 As before, we fix the left-invariant metric on $G$ that makes the basis above  orthonormal, obtaining the following.
 $$g=ds^2+ e^{-2s}dz^2+e^{-s}\left(1+\frac{y^2}4 e^{-s}\right)dx^2+e^{-s}\left(1+\frac{x^2}4 e^{-s}\right)dy^2+\frac{y}2 e^{-2s}dydz-\frac{x}2 e^{-2s}dydz.$$
 
The unitary normal vector on $\mathcal K$ is given by the left-invariant vector field $\nu:=E_0$ and 
 by Lemma \ref{lema1} the mean curvature is given by
 $$H=-2.$$
 
Consider a Killing vector field given by  $V=\gamma \tilde{E}_0$. Imposing Equation \eqref{eq-EvolutionKilling1} we obtain
 $$2= g(\gamma \tilde{E}_0, E_0).$$
 This implies $\gamma=2$ and since
 $\tilde{E}_0$ corresponds to the translation on the left by $(t,0,0,0)$ we have the following translating solution
 $$\mathcal K_t=\{(2t,x,y,z): x,y,z\in \RR\}.$$

Before proceeding with the proof of Theorem \ref{thm:mainintro3}, we recall some basic notions and properties concerning Lie groups and Lie algebras.

Recall that a subgroup $K$ of a Lie group $G$ is normal if and only if 
\[p^{-1}Kp\subset K \mbox{ for all } p\in G.
\] At the Lie algebra level, this is equivalent to $\mathfrak k$ being an ideal of $\mathfrak g$, that is,
$$[X,  \mathfrak k]\subseteq \mathfrak k, \hbox{ for all } X\in  \mathfrak g.$$

Assume henceforth that $\mathfrak k$ is an ideal of $\ggo$. Then, for any $X\in\mathfrak g$, $W,V_1\in \mathfrak k$ and $\lambda \in \mathbb{R}$, we have:

\begin{equation}\label{expad}
e^{-\ad_{\lambda X+W}}(X+V_1)=X + \sum_{j=1}^{\infty} \frac{\ad^j_{-\lambda X-W}(X)}{j!}+V_1+ \sum_{j=1}^{\infty} \frac{\ad^j_{-\lambda X-W}(V_1)}{j!}= X +V_2,
\end{equation} where $V_2\in \mathfrak k$.

Another important property that we will use in the proof is
\begin{equation}\label{Adexp} 
\Ad_{\exp(U)}X=e^{\ad_U}X, \mbox{ for all } U,X\in \mathfrak g,
\end{equation} where the exponential $e$ on the right-hand side is understood as the exponential of linear operators.

We now conclude by proving Theorem \ref{thm:mainintro3}.

\begin{proof}[Proof of Theorem \ref{thm:mainintro3}]

Let $\nu$ be the unit normal of $K$, let $H_0$ denotes the mean curvature of $K$ (which is constant by Lemma \ref{lema1}) and $V$ be the Killing vector corresponding to $H_0 \nu \in \ggo$ given by 
$$V(p)= \frac{d}{dt}\Big|_{t=0} R_p \exp (t H_0 \nu),$$
where $R_p$ denotes right-translation by $p$ and $\exp$ is the Lie exponential map.
Equivalently,
$$V(p)=  (d L_p)_e \circ \Ad_{p^{-1}}(H_0 \nu).$$

Since the exponential map $\exp:\mathfrak g \to G$ is a local diffeomorphism, there exists $\mathcal U\subset \mathfrak g$ and $\mathcal V\subset G$ such that $\exp:\mathcal U \to \mathcal V$ is a diffeomorphism.

Assume $p=\exp(W+\lambda \nu)\in \mathcal V$, with $W\in \mathfrak k$ and $\lambda\in \mathbb R$. Then, 

$$g(V,\nu)=g(\Ad_{\exp(-W-\lambda \nu)}(H_0\nu),\nu)=g(e^{\ad_{-W-\lambda \nu}}(H_0\nu),\nu)=g(e^{\ad_{-W}}H_0\nu,\nu)=H_0,$$ where the last equality follows since 
\[e^{\ad_{-W}}(H_0\nu) = H_0\nu+U, \mbox{ for some } U\in \mathfrak k,\]
and $\nu$ is orthogonal to $\mathfrak{k}$.

To conclude the proof it is enough to show that our assumption on $p$ holds for every point. In fact,
 since $\mathcal V$ generates $G$, every $p\in G$ can be written as

$$p=\exp(W_1+\lambda_1\nu)\exp(W_2+\lambda_2\nu)\hdots \exp(W_n+\lambda_n \nu),$$ for some $n\in \mathbb N$ where $W_i \in\mathfrak k$ and $\lambda_i \in \mathbb R$. Using \eqref{Adexp} repeatedly, we obtain

\begin{align*}
	\Ad_{\exp(W_1+\lambda_1\nu)} {\scriptstyle\circ} \cdots {\scriptstyle\circ} \Ad_{\exp(W_n+\lambda_n\nu)}(H_0\nu)&= \Ad_{\exp(W_1+\lambda_1\nu)} {\scriptstyle\circ} \cdots {\scriptstyle\circ} \Ad_{\exp(W_{n-1}+\lambda_{n-1}\nu)}(H_0\nu +U_n)\\
	&=\Ad_{\exp(W_1+\lambda_1\nu)} {\scriptstyle\circ} \cdots {\scriptstyle\circ}\Ad_{\exp(W_{n-2}+\lambda_{n-2}\nu)}(H_0\nu +U_{n-1})\\
	&= \hspace{1cm}\vdots\\
	&=\Ad_{\exp(W_1 +\lambda_1\nu)} (H_0\nu + U_2)\\
	&=H_0\nu + U_1,
\end{align*} for some $U_i \in \mathfrak k$, which completes the proof.


\end{proof}

\medskip

\noindent{\bf Acknowledgments.}
This project started at the ``Matem\'aticas en el Cono Sur 2'' held at Facultad de Ingenier\'ia (UdelaR), Montevideo, Uruguay, in February 18-23, 2024. \\ The authors would like to thanks Prof. Marcos Salvai (FAMAF - Universidad Nacional de C\'ordoba, Argentina) for fruitful conversations.\\ RA and GO acknowledge Pontificia Universidad Cat\'olica de Chile for hosting them under the program ``Concurso Movilidad de Profesores y Apoyo a Visitantes Extranjeros''.  Finally, we would like to thank the referees for their very careful revision and valuable suggestions.

\section{Declarations}

\noindent{\bf Conflict of interest.} No potential conflict of interest is reported by the authors.

\end{document}